\documentclass[12pt]{article}                                  

\usepackage{amssymb,amsmath,amsthm, amsfonts}
\usepackage{color}
\usepackage{graphicx}

\newtheorem{lemma}{Lemma}
\newtheorem{remark}{Remark}

\newtheorem{definition}{Definition}
\newtheorem{proposition}{Proposition}

\def\dbF{\hbox{\rm l\negthinspace F}}

\def\dbP{\hbox{\rm l\negthinspace P}}

\def\dbR{\hbox{\rm l\negthinspace R}}

\def\cU{{\cal U}}
\def\cV{{\cal V}}

\title{\LARGE \bf
On the Solvability of Risk-Sensitive Linear-Quadratic  Mean-Field Games
}

\author{ Boualem Djehiche \thanks{B. Djehiche is with KTH Royal Institute of Technology, 
        {\tt\small boualem@kth.se}} \,\, and \,\, Hamidou Tembine  \thanks{ H. Tembine is with New York University,
        {\tt\small tembine@nyu.edu}}
}

\begin{document}

\maketitle
\thispagestyle{empty}
\pagestyle{empty}

\begin{abstract}

In this paper we formulate and solve a  mean-field game described by a linear stochastic dynamics and a quadratic or exponential-quadratic cost functional for  each generic player. The optimal strategies for the players are given explicitly using a simple and direct method based on square completion and  a Girsanov-type change of measure, suggested in Duncan {\it et al.} in e.g.  \cite{ref0, ref1} for the mean-field free case. This approach does not use the well-known solution methods such as the Stochastic Maximum Principle and the Dynamic Programming Principle with Hamilton-Jacobi-Bellman-Isaacs equation and Fokker-Planck-Kolmogorov equation. In the risk-neutral linear-quadratic mean-field game, we show that there is unique best response strategy to the mean of the state and provide a simple sufficient condition of existence and uniqueness of mean-field equilibrium. This approach  gives a basic insight into the solution by providing a simple explanation for the additional term in the robust or risk-sensitive Riccati equation, compared to the risk-neutral Riccati equation. Sufficient conditions for existence and uniqueness of mean-field equilibria are obtained when the horizon length  and risk-sensitivity index are small enough. The method is then extended  to the linear-quadratic robust mean-field games under small disturbance, formulated as a minimax  mean-field game.

\end{abstract}

\section{Introduction}
Mean-field games \cite{ref00} with very large number of players have been widely studied recently. Various solution methods such as the Stochastic Maximum Principle (SMP) (\cite{alain2}) and the  Dynamic Programming Principle (DPP) with Hamilton-Jacobi-Bellman-Isaacs equation and Fokker-Planck-Kolmogorov equation have been proposed \cite{ref00,alain2}. Most studies illustrated these solution methods in the linear-quadratic (LQ) game \cite{bardi1,bardi2}. In this paper, we propose a simple argument that gives the best-response strategy and the best-response cost for the LQ-game without use of the well-known solution methods (SMP and  DPP). We apply a simple square completion and a Girsanov-type change of measure (when it applies), successfully applied by Duncan et al. \cite{ref0,ref1,ref2,ref3,ref4} in the mean-field-free case. This method is well suited to LQ games and can hardly be extended to other dynamics and performance functionals. Applying the solution methodology related to the DPP or the SMP requires involved (stochastic) analysis (e.g. in the risk-sensitive case) and convexity arguments to insure necessary and sufficient optimality criteria. We avoid all this with this method. Although, for this simple case, we note that both  DPP and SMP give the same linear structure of the best-response strategies as the actual square completion method.

In the LQ-mean-field game problems the state process can be modeled by a set of linear 
stochastic differential equations of McKean-Vlasov type and the preferences are formalized by quadratic or exponential of integral of quadratic cost functions with mean-field terms. These game  problems are very popular in the literature and a
detailed exposition of this theory can be found in \cite{alain2,alain}. The popularity of these game problems is due to practical considerations
 in signal processing, pattern recognition, filtering and prediction, economics and management science \cite{ref8}.

To some extent, most of the risk-neutral versions of these optimal controls are analytically and numerically solvable \cite{ref1,ref3,ref4}. On the other hand, the linear quadratic risk-sensitive setting naturally appears if the decision makers' objective is to minimize the effect of a small perturbation and related variance of the optimally controlled nonlinear process. By solving a linear quadratic risk-sensitive  game problem, and using the implied optimal control actions, players can significantly reduce the variance (and the cost) incurred by this perturbation. While the risk-sensitive LQ optimal control has been widely investigated in the literature starting from Jacobson 1973 \cite{Jacobson73}, the mean-field version of the problem has been introduced only recently in \cite{Boualem2014,book2}. 
In that paper, the authors established a stochastic maximum principle for  risk-sensitive mean-field-type control where the key mean-field term is the {\it mean state}.  The linear-quadratic risk-sensitive mean-field game has been first introduced in \cite{TAC2014}.

\subsection*{Contribution}
Our contribution can be summarized as follows. In the risk-neutral linear-quadratic mean-field game, we show that there is a unique admissible best-response strategy to any mean-field process. The argument to derive the best response is 
a simple square completion method and is not based on the classical solution methods used in the literature. We derive a fixed-point equation for the mean-field equilibrium, through the state process. The fixed-point equation is shown to have a unique solution on the entire trajectory when the length of the horizon is small enough.  In the risk-sensitive linear-quadratic mean-field game, we establish a risk-sensitive Riccati equation with mean-field term involved in the coefficients. However the existence a positive admissible solution requires a condition on the risk sensitivity index. There is a unique admissible best response strategy when $\theta \sigma^2 <\frac{b^2}{r},$ where $\theta$ is the risk sensitivity index, $b$ is the control action coefficient drift, $r$ is the weight on the quadratic cost and $\sigma$ is the diffusion coefficient. 
Further, we consider robust mean-field games, formulated as a minmax games, and establish a simple connection between the risk-sensitive best and the robust best response. Interestingly, the technique does not require min-max type partial differential equation (such as the Hamilton-Jacobi-Isaacs equation). We also derive a unique admissible best-response strategy under worst-case  disturbance. Finally, a risk-sensitive and robust mean-field game are considered in a similar way.

\medskip
\subsection*{Structure}
 The remainder of the paper is organized as follows. In the next section we present linear-quadratic mean-field games with risk-neutral cost functional. In Section \ref{sec:risk} we focus on risk-sensitive mean-field games. Section \ref{sec:robust} examines robust mean-field games. In Section \ref{sec:robustrisk} we consider the robust risk-sensitive mean-field games. Section \ref{sec:conclusion} concludes the paper.

 \subsection*{Notation and Preliminaries}
Let $T>0$ be a fixed time horizon and $(\Omega,{\cal{F}},\dbF, \dbP)$  be a given filtered
probability space on which a one-dimensional standard
Brownian motion $B=\{B_s\}_{s\geq0}$ is given, and the filtration $\dbF=\{{\mathcal{F}}_s,\ 0\leq s \leq T\}$ is the natural filtration of $B$ augmented by $\dbP-$null sets of ${\cal F}.$

We introduce the following notation.
\begin{itemize} 
\item Let $k\geq 1.$ ${L}^k(0,T;\dbR)$ is the set of functions $f: \ [0,T] \rightarrow \mathbb{R}$  such that $\int_0^T |f(t)|^k dt<\infty$. 
\item $\mathcal{L}^k_{\dbF}(0,T;\dbR)$ is the set of $\dbF$-adapted $\dbR$-valued processes $X(\cdot)$ such that $E\left[\int_0^T |X(t)|^k dt\right]<\infty$. 

\item For $t\in [0,T]$, $|\psi|_t:=\sup_{0\le s\le t}|\psi(s)|.$
\end{itemize}

An admissible control strategy $u$ is an $\dbF$-adapted and square-integrable process with values in a non-empty subset $U$ of $\dbR$. We denote the set of all admissible controls by $\mathcal{U}$:
$$
\mathcal{U}=\{u(\cdot)\in {\cal L}^2_{\dbF}(0,T;\dbR); \,\, u(t)\in U \,\, a.e.\, t\in[0,T],\,\, \dbP-a.s. \}
$$
Given $u\in\mathcal{U}$, consider the following controlled linear SDE.
\begin{equation}
 \left\{\begin{array}{lll}
 dx(t)=[ax(t)+\bar{a}m(t)+bu(t)]dt+\sigma dB(t),\\
x(0)=x_0\in\dbR,
\end{array}
\right.
\end{equation}
where, $a,\bar{a}, b, \sigma$ are real numbers and $m$ is a deterministic function such that $|m|_T<\infty$. Then the following holds

\begin{lemma}\label{L1}
\begin{itemize}
\item[$(1)$] There exists a constant $C_T>0$ depending on $T, a,\bar{a}, b ,\sigma$ and $|m|_T$ such that
\begin{equation}\label{x-est}
E[|x|^2_T]\le C_T\left(1+|x(0)|^2+E\left[\int_0^T |u(t)|^2 dt\right]\right)
\end{equation}
\item[$(2)$] If $f$ is a Borel function from $[0,T]\times \dbR$ into $\dbR$ and is of linear growth i.e. there exists a constant $C>0$ such that $|f(t,x)|\le C(1+|x|)$, for any $(t,x)\in [0,T]\times \dbR$, then the process defined, for $0\le t\le T$, by 
\begin{equation}\label{expo-mg-1}
{\cal E}_t(x):=\exp{\left(\int_0^t f(s,x(s))dB(s)
  -\frac{1}{2}\int_0^t |f(s,x(s))|^2 ds\right)},
\end{equation}
is an $\dbF$-martingale. \\ In particular,
\begin{equation}\label{expo-mag-2}
E\left[{\cal E}_T(x)\right]=1.
\end{equation}
\end{itemize}
\end{lemma}
\begin{proof} The first assertion follows from the Gronwall and the Burkholder-Davis-Gundy's inequalities. The second assertion follows from Corollary 3.5.16 in \cite{karatzas}.  

\end{proof}

\section{ Linear-quadratic mean-field game: risk-neutral case}
Consider a very large number of risk-neutral decision-makers. The best-response of a player is  the following   risk-neutral linear-quadratic mean-field problem where the mean-field term is through the mean of all the states of all the players:
\begin{eqnarray}
\label{LQ00game1} 
\left\{
\begin{array}{lll} \inf_{u(\cdot)\in \cU}  \frac{1}{2}E \left[q(T)x^2(T)+\bar{q}(T)[x(T)-m(T)]^2 \right. \\ \left.
\ \ \ \ \ +\int_0^T q(t) x^2(t)+\bar{q}(t) (x(t)-m(t))^2+r(t) u^2(t)dt \right],
\\
\displaystyle{\mbox{ subject to }\ }\\
dx(t)=[ax(t)+\bar{a}m(t)+bu(t)]dt+\sigma dB(t),\\
x(0)\in \mathbb{R},  \\ 
 \end{array}
\right.
\end{eqnarray}
where, $q(t)\geq 0, \bar{q}(t)\geq 0, r(t)>0,$ and  $a,\bar{a}, b, \sigma$ are real numbers  and where $m(t)$ is the mean state trajectory created by all players at an equilibrium (if it exists). 

Any $\bar u(\cdot)\in {\cal U}$ satisfying the infimum in (\ref{LQ00game1})
is called a risk-neutral best-response strategy of a generic player to the mean-field term $(m(t))_t$. The corresponding state process, solution of (\ref{LQ00game1}), is denoted by $\bar x(\cdot):=x^{\bar u}[m](\cdot).$

The risk-neutral mean-field equilibrium problem we are concerned with is to characterize the triplet $(\bar x,\bar u, m)$ solution of the  problem (\ref{LQ00game1}) and the mean state created by all the players coincide with $m,$ i.e.,
$$m(t)=E[x^{\bar u}[m](t)], $$ which is a fixed-point equation.
\begin{definition}
A  mean-field equilibrium problem is a collection $(\bar x,\bar u, m)$  such that $\bar{u}$ minimizes  (\ref{LQ00game1}) and $E[\bar x]=m.$
\end{definition}
\noindent Note that we  define a  mean-field equilibrium through the mean state and not the entire distribution itself.
\subsection{Determining the best-response of a player}
Let 
\begin{equation} \begin{array}{lll} 
L(u)= \frac{1}{2}[ q(T)x^2(T)+\bar{q}(T)[x(T)-m(T)]^2\\ \qquad\quad +\int_0^T \{ q(t) x^2+\bar{q}(t) (x(t)-m(t))^2+r(t) u^2\} dt].
\end{array}\end{equation} 
Then, the risk-neutral cost functional in (\ref{LQ00game1}) is $E[L(u)]$, the expected value fo $L.$

\noindent Let $f(t,x)=\frac{1}{2}\beta(t) x^2+\alpha(t) x+\gamma(t),$ where, $\alpha, \beta$ and $\gamma$ are deterministic function of time, such that 
$$
f(T,x)=\frac{1}{2} \left[ q(T)x^2+\bar{q}(T)[x-m(T)]^2\right],
$$
i.e. $\beta(T)=q(T)+\bar{q}(T),\,\, \alpha(T)= -\bar{q}(T)m(T),\,\, \gamma(T)= \frac{ \bar{q}(T)}{2}m^2(T).$

\noindent By applying It\^o's formula to $f(t,x(t))$, we have
\begin{eqnarray}&&f(T, x(T))-f(0, x(0)) \nonumber\\ \nonumber
 &=&\int_0^T [f_t(t,x(t))+(ax(t)+\bar{a}m(t)+bu(t))f_x(t,x(t))+\frac{\sigma^2}{2}f_{xx}(t,x(t))]dt \\ \nonumber &&+\int_0^T \sigma f_x (t,x(t))dB(t)\\ \nonumber
 &=&\int_0^T  \{\frac{\dot \beta}{2}(t)+a\beta(t) \} x^2(t)+(\dot{\alpha}(t)+a\alpha(t)+\bar{a}\beta(t) m(t))x(t) dt\\ \nonumber&& 
 +\int_0^T (\dot{\gamma}(t)+\bar{a}\alpha(t) m(t)+\frac{\sigma^2}{2}\beta(t)) +
  bu(\beta(t) x(t)+\alpha(t))
 dt \\ 
 &&+\int_0^T \sigma (\beta(t) x(t)+\alpha(t)) dB(t).
\end{eqnarray}
We evaluate $L(u)-\frac{1}{2}\beta(0) x^2(0)-\alpha(0) x(0)-\gamma(0):$ 
\begin{eqnarray} &&
L(u)-\frac{1}{2}\beta(0) x^2(0)-\alpha(0) x(0)-\gamma(0) \\ \nonumber
&=&  f(T, x(T))-f(0, x(0))\\  \nonumber
 &&+\frac{1}{2} \int_0^T  \{q(t) x^2(t)+\bar{q}(t)(x(t)-m(t))^2+r(t) u^2(t)\} dt\\ \nonumber
&=& \int_0^T  \{\frac{\dot \beta}{2}(t)+a\beta(t) \} x^2(t) dt\\  \nonumber
 && +\int_0^T  (\dot{\alpha}(t)+a\alpha(t)+\bar{a}\beta(t) m(t))x(t) dt\\  \nonumber
 && +\int_0^T  (\dot{\gamma}(t)+\bar{a}\alpha(t) m(t)+\frac{\sigma^2}{2}\beta(t))+
  bu(t)(\beta(t) x(t)+\alpha(t))dt \\ \nonumber
  &&
  +  \int_0^T [\frac{q(t)+\bar{q}(t)}{2} x^2(t)- \bar{q}(t) m(t) x(t)+ \frac{ \bar{q}(t)}{2}m^2(t)+\frac{r(t)}{2}u^2(t) ]\ dt
  \\ \nonumber &&
  +\int_0^T \sigma (\beta(t) x(t)+\alpha(t)) dB(t).
\end{eqnarray}
  
  Thus, 
  \begin{eqnarray} && L(u)-\frac{1}{2}\beta(0) x^2(0)-\alpha(0) x(0)-\gamma(0) \\ \nonumber
   &=& 
   \int_0^T  \{\frac{\dot \beta(t)}{2}+a\beta (t)+\frac{q(t)+\bar{q}(t)}{2}\} x^2 (t)dt\\  \nonumber &&+ \int_0^T(\dot{\alpha}(t)+a\alpha(t)+\bar{a}\beta(t) m(t)- \bar{q}(t) m(t) )x(t) dt\\ \nonumber &&
   +  \int_0^T (\dot{\gamma}(t)+\bar{a}\alpha(t) m(t)+\frac{\sigma^2}{2}\beta(t)+ \frac{ \bar{q}(t)}{2}m^2(t))
  dt \\ \nonumber
  &&
  +  \int_0^T [ bu(t)(\beta(t) x(t)+\alpha(t))+\frac{r(t)}{2} u^2 (t)]\ dt
  \\ &&
  +\int_0^T \sigma (\beta(t) x(t)+\alpha(t)) dB(t)
\end{eqnarray}

We now use the following simple square completion relation:
\begin{eqnarray} &&
 bu(\beta x+\alpha)+\frac{r}{2} u^2\\ \nonumber &=& \frac{r}{2}\left(  u+\frac{b}{r}(\beta x+\alpha)\right)^2-\frac{b^2}{2r}(\beta x+\alpha)^2\\ \nonumber
 &=& \frac{r}{2}\left(  u+\frac{b}{r}(\beta x+\alpha)\right)^2-
 \frac{b^2\beta^2}{2r}x^2-\frac{b^2\alpha\beta}{r}x-\frac{b^2\alpha^2}{2r}.
\end{eqnarray}
Hence,

\begin{eqnarray} &&
L(u)-\frac{1}{2}\beta(0) x^2(0)-\alpha(0) x(0)-\gamma(0)\\ \nonumber
&=&\int_0^T  \{\frac{\dot \beta(t)}{2}+a\beta(t) +\frac{q(t)+\bar{q}(t)}{2}-
 \frac{b^2\beta^2(t)}{2r(t)}\} x^2(t)\\ \nonumber &&+\int_0^T(\dot{\alpha}(t)+a\alpha(t)+\bar{a}\beta(t) m(t)- \bar{q(t)} m(t)-\frac{b^2\alpha(t)\beta(t)}{r(t)} )x(t) dt\\ \nonumber
 &&
   +\int (\dot{\gamma}(t)+\bar{a}\alpha(t) m(t)+\frac{\sigma^2}{2}\beta(t)+ \frac{ \bar{q}(t)}{2}m^2(t)-\frac{b^2\alpha^2(t)}{2r(t)}) 
  dt \\ \nonumber
  &&
  +  \int_0^T   \frac{r(t)}{2}\left[  u(t)+\frac{b}{r(t)}(\beta(t) x(t)+\alpha(t))\right]^2\ dt
  \\ &&
  +\int_0^T \sigma (\beta(t) x(t)+\alpha(t)) dB(t). \label{lbterm}
\end{eqnarray}

Since the expected value of the last term (\ref{lbterm}) is zero and $r>0$, we have
\begin{equation} \label{rn}
E[L(u)]\ge \frac{1}{2}\beta(0) x^2(0)+\alpha(0) x(0)+\gamma(0),\,\,\, u\in\mathcal{U},
\end{equation}
 if and only if
\begin{equation} \label{rn-1}
 \left\{\begin{array}{lll}
\dot \beta(t)+2a\beta(t) -\frac{b^2}{r(t)}\beta^2(t)+q(t)+\bar{q}(t)=0,\\
\beta(T)=q(T)+\bar{q}(T)\geq 0,\\
\dot{\alpha}(t)+a\alpha(t)+(\bar{a}\beta(t) - \bar{q}(t)) m(t)-\frac{b^2}{r(t)}\alpha(t)\beta(t) =0,\\
\alpha(T)= -\bar{q}(T)m(T),\\
\dot{\gamma}(t)+\bar{a}\alpha(t) m(t)+\frac{\sigma^2}{2}\beta(t)+ \frac{ \bar{q}(t)}{2}m^2(t)-\frac{b^2}{2r(t)}\alpha^2(t)=0,\\
\gamma(T)= \frac{ \bar{q}(T)}{2}m^2(T).
\end{array}\right.
\end{equation}
with equality if and only if
\begin{equation*}
 u(t)=-\frac{b}{r(t)}(\beta(t) x(t)+\alpha(t)).
\end{equation*}
Since $b>0, r>0$ by assumption, the risk-neutral Riccati equation in (\ref{rn-1}), has a unique positive solution $\beta.$ Incorporating it into the equation satisfied by $\alpha$ one gets a solution $\alpha(m)$ by direct integration. Thus, the unique best-response strategy is 
\begin{equation}\label{rn-optimal}
 \bar u(t)=-\frac{b}{r(t)}(\beta(t) \bar x(t)+\alpha[m](t)),
\end{equation}
where, $\alpha ,\beta$ are solutions of (\ref{rn-1}). Moreover, the associated performance is
\begin{equation}\label{rn-cost}
E[L(\bar u)]= \frac{1}{2}\beta(0) x^2(0)+\alpha(0) x(0)+\gamma(0).
\end{equation}

\subsection{Risk-Neutral Mean-Field Equilibrium}
We now look for an equilibrium of the risk-neutral linear-quadratic game, which we call risk-neutral mean-field equilibrium.
If $(\bar x, \bar u, m)$ is a mean-field equilibrium then $m$ solves the following fixed-point equation:
\begin{equation}\label{rn-FP}
m(t)=m_0+\int_0^t [(a+\bar{a})m(s)-\frac{b^2}{r(s)} (\beta(s) m(s)+\alpha[m](s))]\ ds=:\Phi[m](t),
\end{equation}
which can be rewritten as $m=\Phi[m].$  
The Banach fixed-point theorem (also known as the contraction mapping theorem) is an important tool in the theory of complete metric spaces. It guarantees the existence and uniqueness of fixed points of certain self-maps of complete metric spaces, and provides a constructive method to find those fixed points by Banach-Picard iterates.

From the inequality in Lemma \ref{L1}, we deduce that the operator $\Phi$ maps ${L}^k(0,T;\dbR)$ into itself, and  ${L}^k(0,T;\dbR)$ is a complete metric space.  Therefore, a simple sufficient condition for having a unique  fixed-point of $\Phi$  is given by a strict contraction of $\Phi,$ which is ensured if its Lipschitz constant $L_{\Phi}$ is strictly less than one.

\noindent A standard Gronwall inequality yields 
$$
L_{\Phi} < T\left[ g_{1}+\tilde g_{1}(\bar{q}_T+\epsilon_1 e^{T | a-b^2\beta/r|_T}) \right]. 
$$ 
where 
\begin{equation}\label{FP-notation}
g_{1}=| a+\bar{a}-b^2\beta/r|_T,\,\,\;  \tilde{g}_{1}= |b^2/r|_T\,\,\;  \epsilon_1=|\bar a\beta-\bar q |_T.
\end{equation}

\medskip\noindent  Thus, we have proved the following 
\begin{proposition}
Suppose that $r>0, s>0, q+\bar{q}>0 $  then there exists a unique best response strategy $\bar u=-\frac{b}{r}(\beta x+\alpha) ,$ 
where $\alpha$ and $\beta$ solve the Riccati equations (\ref{rn-1}).

\noindent In addition, if $T\left[ g_{1}+\tilde g_{1}(\bar{q}_T+\epsilon_1 e^{T | a-b^2\beta/r|_T}) \right] < 1$ then there is a unique risk-neutral mean-field equilibrium.
\end{proposition}

We now refine the equation satisfied by $\alpha$ by setting $\alpha=\eta m.$
\begin{equation} \label{rn-1refined}
 \left\{\begin{array}{lll}
\dot{\alpha}(t)+a\alpha(t)+(\bar{a}\beta(t) - \bar{q}(t)) m(t)-\frac{b^2}{r(t)}\alpha(t)\beta(t) =0,\\
\alpha(T)= -\bar{q}(T)m(T).
\end{array}\right.
\end{equation}
Then, 
\begin{equation} \label{rn-1refined}
 \left\{\begin{array}{lll}
\dot{\eta}+[2a+\bar{a}-2\frac{b^2}{r}\beta]\eta-\frac{b^2}{r}\eta^2+(\bar{a}\beta-\bar{q})=0,\\
\eta(T)= -\bar{q}(T).
\end{array}\right.
\end{equation}
\subsubsection*{Closed-Form Expression of Mean-Field Equilibrium}
Whenever $\alpha$ and $\beta$ does not blow-up within $[0,T],$ the explicit mean-field equilibrium is given by
\begin{equation} \label{rn-1refined2}
 \left\{\begin{array}{lll}
m(t)=m(0) e^{\int_0^t (a+\bar{a}-\frac{b^2}{r}(\beta+\eta))dt}\\
\dot \beta(t)+2a\beta(t) -\frac{b^2}{r(t)}\beta^2(t)+q(t)+\bar{q}(t)=0,\\
\beta(T)=q(T)+\bar{q}(T),\\ \dot{\eta}(t)+[2a+\bar{a}-2\frac{b^2}{r(t)}\beta(t)]\eta(t)-\frac{b^2}{r(t)}\eta^2(t)+(\bar{a}\beta(t)-\bar{q}(t))=0,\\
\eta(T)= -\bar{q}(T)\\
\bar{u}=-\frac{b}{r(t)}(\beta(t) x(t)+\eta(t) m(t))
\end{array}\right.
\end{equation}

\section{ Linear exponential-quadratic mean-field game: risk-sensitive case} \label{sec:risk}
We now consider a very large number of risk-sensitive decision-makers. While the risk-sensitive Linear Quadratic Gaussian optimal control have been widely investigated in the literature starting from Jacobson 1973 \cite{Jacobson73}, the mean-field game version of the problem has been introduced only recently \cite{TAC2014}.

The best-response of a player is  the following risk-sensitive linear-quadratic mean-field problem where the mean-field term is through the mean of all the states of all the players:
\begin{eqnarray}
\label{LQ00gamers} 
\left\{
\begin{array}{lll} \inf_{u(\cdot)\in \cU}  E e^{\theta L(u)},
\\
\displaystyle{\mbox{ subject to }\ }\\
dx(t)=[ax(t)+\bar{a}m(t)+bu(t)]dt+\sigma dB(t),\\
x(0)\in \mathbb{R}.\\ 
 \end{array}
\right.
\end{eqnarray}
Similarly as above, any $\bar u(\cdot)\in {\cal U}$ satisfying the infimum in (\ref{LQ00gamers})
is called a risk-sensitive best-response of a generic player to the mean-field term $(m(t))_t$. The corresponding state process, solution of (\ref{LQ00gamers}), is denoted by $\bar x(\cdot):=x^{\bar u}[m](\cdot).$

The risk-sensitive mean-field equilibrium problem we are concerned with is to characterize the triplet $(\bar x,\bar u, m)$ solution of the  problem (\ref{LQ00gamers}) and the state created by all the players coincide with $m,$ i.e.,
$$m(t)=E[x^{\bar u}[m](t)], $$ which is a  fixed-point equation.

\medskip\noindent We complete with the term  $-\frac{1}{2}\int_0^T \theta^2 \sigma^2 (\beta(t) x(t)+\alpha(t))^2 dt
  +\theta \int_0^T [\frac{1}{2} \theta \sigma^2] (\beta(t) x(t)+\alpha(t))^2 dt$ in $L(u)$ to obtain
\begin{eqnarray} &&
\theta[L(u)-\frac{1}{2}\beta(0) x^2(0)-\alpha(0) x(0)-\gamma(0)]\\ \nonumber 
  &=&
  \theta\int_0^T  \{\frac{\dot \beta}{2}(t)+a\beta(t) +\frac{q(t)+\bar{q}(t)}{2}+(-
 \frac{b^2}{2r(t)}+\frac{1}{2}\theta \sigma^2)\beta^2(t)\} x^2(t) dt\\ \nonumber  &&
 +\theta\int_0^T (\dot{\alpha}(t)+a\alpha(t)+\bar{a}\beta(t) m(t)- \bar{q}(t) m(t)+(-\frac{b^2}{r(t)} +\theta \sigma^2(t))\alpha(t)\beta(t))x(t) dt\\ \nonumber 
 &&
   +\theta \int (\dot{\gamma}(t)+\bar{a}\alpha(t) m(t)+\frac{\sigma^2}{2}\beta(t)+ \frac{ \bar{q}(t)}{2}m^2(t)-\frac{b^2\alpha^2(t)}{2r(t)}+\frac{1}{2}\theta \sigma^2 \alpha^2(t)) 
  dt \\ \nonumber 
  &&
  +  \theta \int_0^T   \frac{r}{2}\left[  u(t)+\frac{b}{r(t)}(\beta(t) x(t)+\alpha(t))\right]^2\ dt
  \\ &&
  +\int_0^T \theta \sigma (\beta(t) x(t)+\alpha(t)) dB(t)
  -\frac{1}{2}\int_0^T \theta^2 \sigma^2 (\beta(t) x(t)+\alpha(t))^2 dt.
\end{eqnarray}
Taking the exponential and the expectation yields 
\begin{equation}\begin{array}{lll}
\mathbb{E}e^{\theta[L(u)-\frac{1}{2}\beta(0) x^2(0)-\alpha(0) x(0)-\gamma(0)]}=
\mathbb{E}\left[{\cal E}_T(x)e^{ \theta \int_0^T \frac{r(t)}{2}\left[  u(t)+\frac{b}{r(t)}(\beta(t) x(t)+\alpha(t))\right]^2  dt}\right]\\ \qquad\qquad\qquad\qquad\qquad\qquad\qquad
    \geq \mathbb{E}\left[{\cal E}_T(x)\right] \end{array}
\end{equation} 
 if and only if $(\beta,\alpha,\gamma)$ satisfies the following system of equations, where we call the equation satisfied by $\beta$ the risk-sensitive Riccati equation:
\begin{equation} \label{rs-1} 
\left\{\begin{array}{ll}
\dot \beta(t)+2a\beta(t)-(
 \frac{b^2}{r(t)}-\theta \sigma^2)\beta^2(t)+q(t)+\bar{q}(t)=0\\
\dot{\alpha}(t)+a\alpha(t)+(\bar{a}\beta(t) - \bar{q}(t) ) m(t)+(-\frac{b^2}{r(t)} +\theta \sigma^2)\alpha(t)\beta(t)=0\\
\dot{\gamma}(t)+\bar{a}\alpha(t) m(t)+\frac{\sigma^2}{2}\beta(t)+ \frac{ \bar{q}(t)}{2}m^2(t)-\frac{b^2\alpha^2(t)}{2r(t)}+\frac{\theta \sigma^2}{2} \alpha^2(t)=0,
\\
\gamma(T)= \frac{ \bar{q}(T)}{2}m^2(T),
\end{array}\right.
\end{equation}
with equality if and only if
\begin{equation}\label{rs-best}
\bar u(t)=-\frac{b}{r(t)}(\beta(t) \bar x(t)+\alpha(t)),
\end{equation}
where, $\bar x$ solves the dynamics in (\ref{LQ00gamers}) associated with $\bar u$, and
\begin{equation*}\label{girsanov}
{\cal E}_t(x):=\exp{\left(\int_0^t \theta \sigma (\beta(s) x(s)+\alpha(s)) dB(s)
  -\frac{1}{2}\int_0^t \theta^2 \sigma^2 (\beta(s) x(s)+\alpha(s))^2 ds\right)}.
\end{equation*}
Now, the risk-sensitive Riccati equation in $\beta$ above has a unique positive solution  if $\frac{b^2}{r}-\theta \sigma^2>0,$ i.e., $\theta < \frac{b^2}{r\sigma^2}$, in which case, in view of Lemma \ref{L1}, we have 
$$
E[{\cal E}_T(x)]=1,\quad u\in\mathcal{U}.
$$
Hence
$$
\mathbb{E}e^{\theta L(u)}\ge e^{\theta \left[\frac{1}{2}\beta(0) x^2(0)+\alpha(0) x(0)+\gamma(0)\right]},\,\, u\in\mathcal{U},
$$
with equality for $u=\bar u$, which constitutes the unique best response to the mean-field process. 

\noindent Hence, the  cost functional associated to the best response $\bar u$ given by (\ref{rs-best}) is 
\begin{equation}
\mathbb{E}e^{\theta L(\bar u)}=e^{\theta\left(\frac{1}{2}\beta(0) x^2(0)+\alpha(0) x(0)+\gamma(0)\right)}.
\end{equation}

\medskip\noindent Note that the presence of the term $\theta \sigma^2$ in the risk-sensitive Riccati equation (\ref{rs-1}) comes from the completion of the exponential martingale. One gets the risk-neutral Riccati equation as $\theta$ vanishes.

\noindent For $T$ and $\theta$ sufficiently small enough, the risk-sensitive mean-field game is completely solvable. Note that, however that for large $\theta,$ the solution $\beta$ may blow-up in finite time and  the control $-\frac{b}{r}(\beta x+\alpha)$ becomes  non-admissible.

\medskip\noindent Thus, we have proved the following 
\begin{proposition} Consider the linear-quadratic risk-sensitive mean-field game associated with (\ref{LQ00gamers}).
Suppose that $r>0, s>0, q+\bar q>0$  and $\theta < \frac{b^2}{r\sigma^2}$ then there exists a unique best response strategy $\bar u=-\frac{b}{r}(\beta x+\alpha),$ 
where $\alpha$ and $\beta$ solves the risk-sensitive Riccati equation (\ref{rs-1}) .

 In addition, if $T\left[ g_{2}+\tilde g_{2}(\bar{q}_T+\epsilon_2 e^{T | a+(-b^2/r+\theta\sigma^2)\beta |_{T}}) \right] < 1$ then there is a unique robust  risk-sensitive mean-field equilibrium, where 
$g_{2}=| a+\bar{a}+(-b^2/r)\beta |_{T},\ \tilde{g}_{2}=|  -b^2/r|_T,\  \epsilon_2=|\bar a\beta-\bar q |_{T},$ which are modification of $g_{1}, \tilde{g}_{1}$ and $\epsilon_1$  given by (\ref{FP-notation}).

\end{proposition}

\section{ Robust mean-field game: the risk-neutral case} \label{sec:robust}

Consider a very large number of risk-neutral players and a malicious/disturbance term $v$. We refer to \cite{robust} for an interesting application of robust mean-field games to crowd seeking problems in social networks. We assume that the disturbance term $v$ takes values a subset $V$ of $\dbR$ and define the set of disturbance strategies  $\mathcal{V}$ in a similar fashion as $\mathcal{U}$. Next, we formulate a  risk-neutral robust mean-field game as a minmax mean-field game as follows.

\noindent Let
\begin{equation}\label{r-rs} 
\begin{array}{lll}
L_2(u,v)= \frac{1}{2}\int_0^T \{ q(t)x^2(t)+\bar{q} (t)(x(t)-m(t))^2+r(t) u^2(t)-s(t) v^2(t)\} dt \\ \qquad\quad +\frac{1}{2}\left[q(T)x^2(T)+\bar{q}(T)[x(T)-m(T)]^2\right]
\end{array}
\end{equation} 
be the  cost functional of a generic player with strategy $u$ under disturbance $v$ when the mean-field process is $m.$

The best-response of a generic player is  the following   risk-neutral linear-quadratic robust mean-field problem under worst case disturbance  is 
\begin{eqnarray}
\label{LQ00game2} 
\left\{
\begin{array}{lll} \inf_{u(\cdot)\in \cU} \sup_{v(\cdot)\in \cV} E \left[ L_2(u,v)\right],
\\
\displaystyle{\mbox{ subject to }\ }\\
dx(t)=[ax+\bar{a}m(t)+bu(t)+c v(t)]dt+\sigma dB(t),\\
x(0)\in \mathbb{R},\\ 
 \end{array}
\right.
\end{eqnarray}
where, $q(t)\geq 0, \bar{q}(t)\geq 0, r(t)>0,$ and  $a,\bar{a}, b, \sigma$ are real numbers  and where $m(t)$ is the mean state trajectory created by all players at robust equilibrium (if it exists). 

Any $(\bar u(\cdot), \bar v(\cdot))\in {\cal U}\times {\cal V}$ satisfying the min-max in (\ref{LQ00game2})
is called a risk-neutral robust best-response of a generic player to the mean-field process $(m(t))_t$ under worst case disturbance $\bar v$. The corresponding state process, solution of (\ref{LQ00game2}), is denoted by $\bar x(\cdot):=x^{\bar u, \bar v}[m](\cdot).$

The robust mean-field equilibrium problem we are concerned with is to characterize the collection $(\bar x,\bar u, \bar v, m)$ solution of the  problem (\ref{LQ00game2}) and the mean state created by all the players coincide with $m,$ i.e.,
$$m(t)=E[x^{\bar u, \bar v}[m](t)], $$ which is a fixed-point equation.
\begin{definition}
A  robust mean-field equilibrium problem is a collection $(\bar x,\bar u, \bar v, m)$  such that $\bar{u}$ minimizes  (\ref{LQ00game2}) under worst disturbance $\bar v$ and $E[\bar x]=m.$
\end{definition}

At a robust mean-field equilibrium, the best-response to the mean-field under worst case disturbance,  should reproduce  the mean-field  itself.

\subsection{Determining the robust best-response of a player}
The terminal cost is similar as above but now It\^o's formula to  the function $f$ is  different because of the disturbance.
\noindent Let $f(t,x)=\frac{1}{2}\beta(t) x^2(t)+\alpha(t) x(t)+\gamma(t).$ By applying It\^o's formula, we have
\begin{eqnarray}&&f(T, x(T))-f(0, x(0)) \nonumber \\  \nonumber 
 &=&\int_0^T [f_t(t,x(t))+(ax(t)+\bar{a}m(t)+bu(t)+cv(t))f_x(t,x(t))+\frac{\sigma^2}{2}f_{xx}(t,x(t))]dt \\ \nonumber && +\int_0^T \sigma f_x (t,x(t))dB(t)\\ \nonumber 
 &=&\int_0^T  \{\frac{\dot \beta(t)}{2}+a\beta(t) \} x^2(t)+(\dot{\alpha}(t)+a\alpha(t)+\bar{a}\beta(t) m(t))x(t) dt\\ \nonumber  &&
 +\int_0^T (\dot{\gamma}(t)+\bar{a}\alpha(t) m(t)+\frac{\sigma^2}{2}\beta(t)) +
  (bu(t)+cv(t))(\beta(t) x(t)+\alpha(t))
 dt \\ && +\int_0^T \sigma (\beta(t) x(t)+\alpha(t)) dB(t).
\end{eqnarray}
We now compute the difference 
$L_2(u,v)-\frac{1}{2}\beta(0) x^2(0)-\alpha(0) x(0)-\gamma(0).$

\noindent We have
  \begin{equation} \begin{array}{lll} 
L_2(u,v)-\frac{1}{2}\beta(0) x^2(0)-\alpha(0) x(0)-\gamma(0)= 
   \int_0^T \{\frac{\dot \beta(t)}{2}+a\beta(t) +\frac{q(t)+\bar{q}(t)}{2}\} x^2(t)\ dt \\ \qquad\qquad \qquad   +\int_0^T (\dot{\alpha}(t)+a\alpha(t)+\bar{a}\beta(t) m(t)- \bar{q}(t) m(t) )x(t) dt\\ \qquad\qquad\qquad  
   +  \int_0^T (\dot{\gamma}(t)+\bar{a}\alpha(t) m(t)+\frac{\sigma^2}{2}\beta(t)+ \frac{ \bar{q}(t)}{2}m^2(t))
  dt \\  \qquad\qquad\qquad  
  +  \int_0^T [ (bu(t)+cv(t))(\beta(t) x(t)+\alpha(t))+\frac{r(t)}{2} u^2(t)-\frac{s(t)}{2} v^2 (t)]\ dt
  \\ \qquad\qquad\qquad  
  +\int_0^T \sigma (\beta(t) x(t)+\alpha(t)) dB(t).
\end{array}
\end{equation}
We now use the following relations:
\begin{eqnarray*} 
   bu(\beta x+\alpha)+\frac{r}{2} u^2= \frac{r}{2}\left(  u+\frac{b}{r}(\beta x+\alpha)\right)^2 -
 \frac{b^2\beta^2}{2r}x^2-\frac{b^2\alpha\beta}{r}x-\frac{b^2\alpha^2}{2r}
\end{eqnarray*}
and
\begin{eqnarray*} 
 cv(\beta x+\alpha)-\frac{s}{2} v^2= -\frac{s}{2}\left(  v-\frac{c}{s}(\beta x+\alpha)\right)^2+
 \frac{c^2\beta^2}{2s}x^2-\frac{c^2\alpha\beta}{r}x-\frac{c^2\alpha^2}{2s}
 \end{eqnarray*}
 to obtain
\begin{eqnarray} &&
L_2(u,v)-\frac{1}{2}\beta(0) x^2(0)-\alpha(0) x(0)-\gamma(0)  \nonumber \\ \nonumber
&=&\int_0^T  \{\frac{\dot \beta(t)}{2}+a\beta(t) +\frac{q(t)+\bar{q}(t)}{2}+(-
 \frac{b^2}{2r(t)}+ \frac{c^2}{2s(t)})\beta^2(t)\} x^2(t)\\  \nonumber &+&\int_0^T (\dot{\alpha}(t)+a\alpha(t)+\bar{a}\beta(t) m(t)- \bar{q}(t) m(t)+(-\frac{b^2}{r(t)}+\frac{c^2}{s(t)})\alpha(t)\beta(t) )x(t) dt\\ \nonumber 
 & +&
\int (\dot{\gamma}(t)+\bar{a}\alpha(t) m(t)+\frac{\sigma^2}{2}\beta(t)+ \frac{ \bar{q}(t)}{2}m^2(t)+(-\frac{b^2}{2r(t)}+\frac{c^2}{2s(t)}) \alpha^2(t)) \nonumber 
  dt \\  \nonumber 
  & +&
   \int_0^T   \frac{r(t)}{2}\left[  u(t)+\frac{b}{r(t)}(\beta(t) x(t)+\alpha(t))\right]^2-\frac{s(t)}{2}\left[  v(t)-\frac{c}{s(t)}(\beta(t) x(t)+\alpha(t))\right]^2\ dt
  \\   &+&
  \int_0^T \sigma (\beta(t) x(t)+\alpha(t)) dB(t).
\end{eqnarray}
Therefore,
\begin{equation} \label{r-rn}
 E[L_2(u,v)]\ge\frac{1}{2}\beta(0) x^2(0)+\alpha(0) x(0)+\gamma(0),\,\, (u,v)\in\mathcal{U}\times\mathcal{V},
\end{equation}
if and only if  $(\alpha, \beta, \gamma)$ solves the following system of equations where we call the equation $\beta$ solves, 'robust' Riccati equation:
\begin{equation}\label{r-rn-1}\left\{\begin{array}{lll}
\dot \beta(t)+2a\beta(t) +(-\frac{b^2}{r(t)}+ \frac{c^2}{s(t)})\beta^2(t)+q(t)+\bar{q}(t)=0,\\
\beta(T)=q(T)+\bar{q}(T)\geq 0,\\
\dot{\alpha}(t)+a\alpha(t)+(\bar{a}\beta(t)- \bar{q}(t)) m(t)+(-\frac{b^2}{r(t)}+\frac{c^2}{s(t)})\alpha(t)\beta(t)=0,\\
\alpha(T)= -\bar{q}(T)m(T),\\
\dot{\gamma}(t)+\bar{a}\alpha(t) m(t)+\frac{\sigma^2}{2}\beta (t)+ \frac{ \bar{q}(t)}{2}m^2(t)+(-\frac{b^2}{2r(t)}+\frac{c^2}{2s(t)}) \alpha^2(t)\\
\gamma(T)=\frac{ \bar{q}(T)}{2}m^2(T),
\end{array}
\right.
\end{equation}
 with equality if and only if 
\begin{equation}
\bar u(t)=-\frac{b}{r(t)}(\beta(t) x(t)+\alpha(t)),\quad 
\bar v(t)=\frac{c}{s(t)}(\beta(t) x(t)+\alpha(t)).
\end{equation}

Since $b>0, r>0, s>0$ by assumption, the robust Riccati  equation (\ref{r-rn-1}) in $\beta$ has a unique positive solution if  $\frac{b^2}{r}- \frac{c^2}{s}>0$.  Injecting $\beta$ into the equation satisfied by $\alpha$ we get a solution $\alpha(m)$ by direct integration. Thus, there exists a unique best-response strategy whenever $\frac{b^2}{r}> \frac{c^2}{s}>0.$

\begin{remark} Setting $c^2/s=\theta \sigma^2$ in the robust Riccati equation  (\ref{r-rn-1}), we obtain  the  risk-sensitive Riccati equation (\ref{rs-1}).  
\end{remark}

\subsection{Risk-Neutral Robust Mean-Field Equilibrium}
If $(\bar x, \bar u,\bar v, m)$ is a robust mean-field equilibrium then $m$ solves the following fixed-point equation:
\begin{equation}\label{FP-2}
m=\Phi[m],
\end{equation}
where,
\begin{equation*}
\Phi[m](t):=m_0+\int_0^t (a+\bar{a})m(t')+(-\frac{b^2}{r(t')}+\frac{c^2}{s(t')}) (\beta(t') m(t')+\alpha[m](t'))dt'.
\end{equation*}
Since the term $(-\frac{b^2}{r}+\frac{c^2}{s})$ is changed compared to the previous analysis and  $$L_{\Phi} < T\left[ g_{3T}+\tilde g_{3T}(\bar{q}_T+\epsilon_3 e^{T | a+(-b^2/r+c^2/s)\beta|_{T}}) \right],$$
where, $g_{3}=| a+\bar{a}+(-b^2/r+c^2/s)\beta|_{T},\ \tilde{g}_{3}=|  (-b^2/r+c^2/s)|_T,\  \epsilon_3=|\bar a\beta-\bar q |_{T}.$
Thus, we have proved the following 
\begin{proposition} Consider the robust linear-quadratic game associated with (\ref{LQ00game2}).
Suppose that $r>0, s>0, $ and $\frac{b^2}{r}> \frac{c^2}{s}>0$ then there exists a unique best response strategy $\bar u=-\frac{b}{r}(\beta x+\alpha),$ and the worst case disturbance is $\bar v= \frac{c}{s}(\beta x+\alpha).$
where $\alpha$ and $\beta$ solves the robust Riccati equation (\ref{r-rn-1}).

\noindent In addition, if $T\left[ g_{3T}+\tilde g_{3T}(\bar{q}_T+\epsilon_3 e^{T | a+(-b^2/r+c^2/s)\beta|_{T}}) \right] < 1, $  then there is a unique   risk-neutral robust mean-field equilibrium.
\end{proposition}

\section{ Linear exponential-quadratic robust mean-field game} \label{sec:robustrisk}

The risk-sensitive robust  best-response of a player is  a solution to the following minmax problem:
\begin{eqnarray}
\label{LQ00gamers2} 
\left\{
\begin{array}{lll} \inf_{u(\cdot)\in \cU} \sup_{v(\cdot)\in \cV}  E e^{\theta L_2(u,v)},
\\
\displaystyle{\mbox{ subject to }\ }\\
dx(t)=[ax(t)+\bar{a}m(t)+bu(t)+cv(t)]dt+\sigma dB(t),\\
x(0)\in\mathbb{R}.\\ 
 \end{array}
\right.
\end{eqnarray}

Similarly as above, any $(\bar u(\cdot),\bar v(\cdot))\in {\cal U}\times  {\cal V}$ satisfying the min-max in (\ref{LQ00gamers2})
is called a risk-sensitive robust best-response of a generic player to the mean-field process $(m(t))_t$ under worst case disturbance $v$. The corresponding state process, solution of (\ref{LQ00gamers2}), is denoted by $\bar x(\cdot):=x^{\bar u, \bar v}[m](\cdot).$

 The risk-sensitive robust mean-field equilibrium problem we are concerned with is to characterize the pair $(\bar x,\bar u, \bar v)$ solution of the  problem (\ref{LQ00gamers2}) and the state created by all the players coincide with $m$, i.e.,
$$m(t)=E[x^{\bar u, \bar v}[m](t)], $$ which is a fixed-point equation. Completing with the term  
$$
-\frac{1}{2}\int_0^T \theta^2\sigma^2 (\beta(t) x(t)+\alpha(t))^2 dt
  +\frac{1}{2}\theta \int_0^T \theta \sigma^2 (\beta(t) x(t)+\alpha(t))^2 dt
$$ 
in $L_2(u,v)$ we obtain
\begin{eqnarray} &&
\theta[L_2(u,v)-\frac{1}{2}\beta(0) x^2(0)-\alpha(0) x(0)-\gamma(0)]\\ \nonumber
&=&
\theta \int_0^T  \{\frac{\dot \beta(t)}{2}+a\beta(t) +\frac{q(t)+\bar{q}(t)}{2}+(-
 \frac{b^2}{2r(t)}+ \frac{c^2}{2s(t)}+\frac{\theta}{2}\sigma^2)\beta^2(t)\} x^2(t) dt \\ \nonumber
 &+& \theta \int (\dot{\alpha}(t)+a\alpha(t)+\bar{a}\beta(t) m(t)- \bar{q}(t) m(t)+(-\frac{b^2}{r(t)}+\frac{c^2}{s(t)}+\theta\sigma^2)\alpha(t)\beta(t) )x(t) dt\\ \nonumber
 &+&
   \theta \int (\dot{\gamma}(t)+\bar{a}\alpha(t) m(t)+\frac{\sigma^2}{2}\beta(t)+ \frac{ \bar{q}(t)}{2}m^2(t)+(-\frac{b^2}{2r(t)}+\frac{c^2}{2s(t)}+\frac{\theta}{2}\sigma^2) \alpha^2(t))
  dt \\ \nonumber
  &+&
   \theta  \int_0^T   \frac{r(t)}{2}\left[  u+\frac{b}{r(t)}(\beta(t) x(t)+\alpha(t))\right]^2-\frac{s(t)}{2}\left[  v(t)-\frac{c}{s(t)}(\beta(t)x(t)+\alpha(t))\right]^2\ dt
  \\ &+&
  \int_0^T \theta \sigma (\beta(t) x(t)+\alpha(t)) dB(t)-
  \frac{1}{2}\int_0^T \theta^2 \sigma^2 (\beta(t) x(t)+\alpha(t))^2 dt.
\end{eqnarray}
Taking exponential and the expectation yields
\begin{eqnarray} \nonumber 
\mathbb{E}e^{\theta[L_2(u,v)]}\ge \mathbb{E}\left[ {\cal E}_T(x) \right]\exp{\theta\left[\frac{1}{2}\beta(0) x^2(0)+\alpha(0) x(0)+\gamma(0)\right]},\quad (u,v)\in \mathcal{U}\times\mathcal{V},
\end{eqnarray}
where,
$$
{\cal E}_T(x):=\exp{\left( \int_0^T \theta \sigma (\beta(t) x(t)+\alpha(t)) dB(t)-
  \frac{1}{2}\int_0^T \theta^2 \sigma^2 (\beta(t) x(t)+\alpha(t))^2 dt\right)},
$$
if and only if 
\begin{equation}\label{r-rs-1}
\left\{\begin{array}{lll}
\dot \beta(t)+2a\beta(t) +(-\frac{b^2}{r(t)}+ \frac{c^2}{s(t)}+\theta \sigma^2)\beta^2(t)+q(t)+\bar{q}(t)=0,\\
\beta(T)=q(T)+\bar{q}(T)\geq 0,\\
\dot{\alpha}(t)+a\alpha(t)+(\bar{a}\beta(t)- \bar{q}(t)) m(t)\\ \ \ +(-\frac{b^2}{r(t)}+\frac{c^2}{s(t)}+\theta \sigma^2)\alpha(t)\beta(t)=0,\\
\alpha(T)= -\bar{q}(T)m(T),\\
\dot{\gamma}(t)+\bar{a}\alpha(t) m(t)+\frac{\sigma^2}{2}\beta(t)+ \frac{ \bar{q}(t)}{2}m^2(t)+(-\frac{b^2}{2r(t)}+\frac{c^2}{2s(t)}+\frac{\theta}{2}\sigma^2) \alpha^2(t),\\
\gamma(T)=\frac{ \bar{q}(T)}{2}m^2(T),
\end{array}\right.
\end{equation}
with equality if and only if
\begin{equation}\label{r-rs-opt}
\bar u(t)=-\frac{b}{r(t)}(\beta (t) x(t)+\alpha(t)), \,\,\,
\bar v(t)=\frac{c}{s(t)}(\beta(t) x(t)+\alpha(t)).
\end{equation}

The risk-sensitive robust Riccati equation in $\beta$ above has a unique positive solution $\beta(t)\geq 0$ if $
 \frac{b^2}{r}- \frac{c^2}{s}-\theta \sigma^2>0,$  which is satisfied if $\theta$ and $c$ are  relatively small enough. In this case, in view of Lemma \ref{L1}, $E[{\cal E}_T(x)]=1$, which yields
\begin{equation} \nonumber 
\mathbb{E}e^{\theta[L_2(u,v)]}\ge e^{\theta\left[\frac{1}{2}\beta(0)x^2(0)+\alpha(0) x(0)+\gamma(0)\right]}, \quad (u,v)\in \mathcal{U}\times\mathcal{V},
\end{equation}
with equality if $(u, v)=(\bar u,\bar v),$ which  is a unique  best response strategy to the mean-field for  $\frac{b^2}{r}> \frac{c^2}{s}+\theta \sigma^2>0$. 

\noindent For $T$ and $\theta, c$ sufficiently small enough, the risk-sensitive mean-field game is completely solvable. Note however that for large $\theta,$ or $c,$ the solution $\beta$ may blow-up in finite time and  the control $-\frac{b}{r}(\beta x+\alpha)$ becomes  non-admissible.

\noindent Thus, we have proved the following 
\begin{proposition}
Suppose that $r>0, s>0, q>0, \bar{q}\geq 0$ and $\frac{b^2}{r}> \frac{c^2}{s}+\theta \sigma^2>0$ then there exists a unique best response strategy $\bar u=-\frac{b}{r}(\beta x+\alpha)$ and the worst case disturbance is $\bar v= \frac{c}{s}(\beta x+\alpha).$
where $\alpha$ and $\beta$ solves the risk-sensitive Riccati equations (\ref{r-rs-1}).

\noindent 
In addition, if $T\left[ g_{4}+\tilde g_{4}(\bar{q}+\epsilon_4 e^{T | a+(-b^2/r+c^2/s+\theta \sigma^2)\beta|_{T}}) \right] < 1$ then there is a unique robust  risk-sensitive mean-field equilibrium, where,
$g_{4}=| a+\bar{a}+(-b^2/r+c^2/s)\beta|_{T},\ \tilde{g}_{4}=|  (-b^2/r+c^2/s)|_T,\  \epsilon_4=|\bar a\beta-\bar q |_{T}.$
 \end{proposition}

\begin{remark}
If the mean-field term is  a function of the equilibrium control action  $\xi(E[\bar u])$ instead of the mean state $\bar m$, the same methodology applies, except that the fixed-point equation  is now different $$E[\bar u]=-\frac{\bar{b}}{r}\left[ \beta \bar m+\alpha[ \xi(E[\bar u])]\right].$$  This type of mean-field structures is observed in smart grids where the mean-field term which modulates the price of electricity is a function of aggregate demand and supply. The demand is generated by the users (residential consumers, commercial, industrial, transportation etc). It is also relevant in wireless networking where the mean-field term is the interference  $E[\bar u.|\bar x|^2]$   created by other users and $\bar u\geq 0$ is the transmission power of the user. See \cite{wireless} for more details.
\end{remark}

\section{Concluding remarks} \label{sec:conclusion}
The method described in this paper provides an alternative to both the Hamilton-Jacobi-Bellman-Isaacs equation and the robust stochastic maximum principle for the particular case of LQ-games. The method  exhibits the best-response strategy  of the player and the best-response cost directly in the problem solution. The mean-field equilibrium is then formulated as a fixed-point of the best-response. Sufficient conditions for existence and uniqueness of mean-field equilibria are provided. The approach can be extended to more general LQ-games in several ways: (i) matrix form, (ii)  games with anomalous diffusions, (iii) partially observable games, (iv)    robust cooperative mean-field-type games. We leave these questions  for future work. We note that this method can hardly be extended to games with arbitrary  nonlinear dynamics and cost performance.

\end{document}